\documentclass[12pt,a4paper]{article}
\usepackage{amsmath,amsfonts,amssymb,amsthm}
\usepackage{mathrsfs}
\usepackage{ifthen,textcomp,multirow,longtable,array}
\usepackage[colorlinks=true,citecolor=magenta,linktocpage=true,linkcolor=blue]{hyperref}
\usepackage[twoside,a4paper, mag=1000,left=2.5cm, right=2.5cm, top=2cm, bottom=2cm, headsep=0.7cm, footskip=1cm]{geometry}

\newtheorem{Theorem}{Theorem}[section]

\newtheorem{Remark}[Theorem]{Remark}

\newcommand\emp\varnothing
\newcommand\eps\varepsilon
\newcommand\ov\widetilde

\def\^#1{^{\overline{#1}}}
\newcommand{\email}[1]{\href{mailto:#1}{#1}}

\newcommand{\Addresses}{{
  \bigskip
  \footnotesize

  O.\,P. Khromova,\par\nopagebreak \textsc{Altai State University}\par\nopagebreak \textsc{Barnaul, Russia}\par\nopagebreak
  \textit{E-mail address:} \email{khromova.olesya@gmail.com}

  \medskip

  S.\,V. Klepikova,\par\nopagebreak \textsc{Altai State University}\par\nopagebreak \textsc{Barnaul, Russia}\par\nopagebreak
  \textit{E-mail address:} \email{klepikova.svetlana.math@gmail.com}
  
  \medskip

  E.\,D. Rodionov,\par\nopagebreak \textsc{Altai State University}\par\nopagebreak \textsc{Barnaul, Russia}\par\nopagebreak
  \textit{E-mail address:} \email{edr2002@mail.ru}

}}

\begin{document}

\title{On Sectional Curvature Operator \\of 3-dimensional Locally Homogeneous \\Lorentzian Manifolds\thanks{The reported study was funded by RFBR according to the research project \textnumero~18--31--00033 mol\_a.}}

\author{O.\,P. Khromova, S.\,V. Klepikova, E.\,D. Rodionov \\
}


\date{}

\maketitle

\begin{abstract}
The main purpose of this paper is to determine the admissible forms of the~sectional curvature operator on a three-dimensional locally homogeneous Lorentzian manifolds. 
\end{abstract}

\allowdisplaybreaks

\section{Introduction}

\sloppy
The problem of the restoring of (pseudo)Riemannian manifold on prescribed curvature operator is the actual direction in the research of curvature operators. The Riemannian locally homogeneous spaces with the prescribed values of the spectrum of Ricci operator have been identified by O.\,Kowalski and S.\,Nikcevic in~\cite{KlepikovaSV:Koval-Ni1996}. The problem of~the existence of~locally homogeneous Lorentzian manifold and prescribed Ricci operator was investigated by G.\,Calvaruso and O.\,Kowalski~in~\cite{KlepikovaSV:CalKow2009}. There are also some papers about this problem in nonhomogeneous case (see~\cite{KlepikovaSV:Calvaruso2008,KlepikovaSV:Kowalski1993}).

Similar results were obtained by D.N.~Oskorbin, E.D.~Rodionov, O.P. Khromova for~the~one-dimensional curvature operator and the sectional curvature operator in~the~case of three-dimensional Lie groups with left-invariant Riemannian metrics~\cite{KlepikovaSV:OR2013,KlepikovaSV:GO2013}.


The main purpose of this paper is to consider the problem of the precribed  sectional curvature operator $\mathcal{K}$ on the three-dimensional Lorentzian locally homogeneous manifolds. 

Unlike the case of the Riemannian metric, there always exist an orthonormal basis, in~which the matrix of the curvature operator is diagonal, in the case of Lorentzian metric different cases can occur known as \textit{Segre types} (see.~\cite{KlepikovaSV:Segre_types_2000}). Namely, the following cases can~occur:
\begin{enumerate}
\item Segre type $\{111\}$: the operator $\mathcal{K}$ has three real eigenvalues (possibly coincident), each associated to a one-dimensional eigenspace.
\item Segre type $\{1z\overline{z}\}$: the operator $\mathcal{K}$ has one real and two complex the conjugate eigenvalues.
\item Segre type $\{21\}$: the operator $\mathcal{K}$ has two real eigenvalues (possibly coincident), the first of which has algebraic multiplicity 2, each associated to a one-dimensional eigenspace.
\item Segre type $\{3\}$: the operator $\mathcal{K}$ has one real eigenvalue of algebraic multiplicity 3, associated one-dimensional eigenspace.
\end{enumerate}

\section{Three-dimensional homogeneous Loretzian manifolds}


Let $(M,g)$ be three-dimensional homogeneous manifold, with the Lorenzian metric $g$ of~signature $\left(+,+,-\right)$. We denote by $\nabla$ its Levi-Civita connection and by $R$ its curvature tensor, which defined by following
$$R\left(X,Y\right)Z=\left[\nabla_Y,\nabla_X\right]Z+\nabla_{\left[X,Y\right]}Z.$$ 

The Lorenzian metric $g$ induces a scalar product $\langle\cdot,\cdot\rangle$ in the bundle $\Lambda^2 M$ by the rule
$$ \left\langle X_1 \wedge X_2, Y_1 \wedge Y_2 \right\rangle = \mathrm{det}\left(\left\langle X_i,Y_i\right\rangle\right). $$
The curvature tensor $R$ at any point can be considered as an operator  $\mathcal{K} \colon \Lambda^2 M \to \Lambda^2 M$,
called the \textit{sectional curvature operator} and defined by the equation
\begin{equation*}\label{eq:def_sec_curv}
\left\langle X \wedge Y,\mathcal{K}\left(Z \wedge T\right)\right\rangle = R\left(X, Y, Z, T\right).
\end{equation*}

The studying of curvature operators on three-dimensional locally homogeneous Lorentzian spaces is based on the following fact, which was proved by G.\,Calvaruso in~\cite{KlepikovaSV:Calvaruso2007}.

\begin{Theorem}
Let $(M,g)$ be a three-dimensional connected, simply connected, complete locally homogeneous Lorentzian manifold. Then, either $(M,g)$ is locally symmetric, or it is~locally isometric to a three\nobreakdash-dimensional Lie group equipped with a left-invariant Lorentzian metric.
\end{Theorem}

Further classification results for three-dimensional Lorentzian Lie groups was obtained~in~\cite{KlepikovaSV:RSCh2006}. 


\begin{Theorem}
Let $G$ be a three-dimensional Lie group with left-invariant Lorentzian metric. Then
\begin{itemize}
\item If $G$ is unimodular, then there exists a pseudo-orthonormal frame field $\left\{e_1, e_2, e_3\right\}$, such that the metric Lie algebra of $G$ is one of the following:
\begin{enumerate}
\item
\begin{equation*}
\mathcal{A}_1:
\begin{split}
[e_1,e_2] &= \lambda_3e_3, \\
[e_1,e_3] &= -\lambda_2e_2, \\
[e_2,e_3] &= \lambda_1e_1,
\end{split}
\end{equation*}
with $e_1$ timelike;
\item
\begin{equation*}
\mathcal{A}_2:
\begin{split}
[e_1,e_2] &= \left(1-\lambda_2\right)e_3-e_2, \\ 
[e_1,e_3] &= e_3-\left(1+\lambda_2\right)e_2, \\
[e_2,e_3] &= \lambda_1e_1,
\end{split}
\end{equation*}
with $e_3$ timelike;
\item
\begin{equation*}
\mathcal{A}_3:
\begin{split}
[e_1,e_2] &= e_1-\lambda e_3, \\ 
[e_1,e_3] &= -\lambda e_2-e_1, \\
[e_2,e_3] &= \lambda_1e_1+e_2+e_3,
\end{split}
\end{equation*}
with $e_3$ timelike;
\item
\begin{equation*}
\mathcal{A}_4:
\begin{split}
[e_1,e_2] &= \lambda_3e_2, \\ 
[e_1,e_3] &= -\beta e_1-\alpha e_2, \\
[e_2,e_3] &= -\alpha e_1+\beta e_2,
\end{split}
\end{equation*}
with $e_1$ timelike and $\beta \ne 0$.
\end{enumerate}
\item If $G$ is non-unimodular, then there exists a pseudo-orthonormal frame field $\left\{e_1, e_2, e_3\right\}$, such that the metric Lie algebra of $G$ is one of~the~following:
\begin{enumerate}
\item
\begin{equation*}
\mathcal{A}:
\begin{split}
[e_1,e_2] &= 0, \\ 
[e_1,e_3] &= \lambda \sin\varphi\,e_1-\mu\cos\varphi\,e_2, \\
[e_2,e_3] &= \lambda \cos\varphi\,e_1+\mu\sin\varphi\,e_2,
\end{split}
\end{equation*}
with $e_3$ timelike and $\sin\varphi\ne 0$, $\lambda+\mu \ne 0$, $\lambda \geqslant 0$, $\mu \geqslant 0$;
\item
\begin{equation*}
\mathcal{B}: 
\begin{split}
[e_1,e_2] &= 0, \\ 
[e_1,e_3] &= te_1-se_2, \\
[e_2,e_3] &= pe_1+qe_2,
\end{split}
\end{equation*}
with $\left\langle e_2,e_2\right\rangle=-\left\langle e_1,e_3\right\rangle=1$ and otherwise zero, and $q \ne t$;
\item
\begin{equation*}
\mathcal{C}_1: 
\begin{split} 
[e_1,e_2] &= 0, \\ 
[e_1,e_3] &= se_1+pe_2, \\
[e_2,e_3] &= pe_1+qe_2,
\end{split}
\end{equation*}
with $e_2$ timelike and $q\ne s$;
\item
\begin{equation*}
\mathcal{C}_1: 
\begin{split} 
[e_1,e_2] &= 0, \\ 
[e_1,e_3] &= qe_1-re_2, \\
[e_2,e_3] &= pe_1+qe_2,
\end{split}
\end{equation*}
with $e_2$ timelike and $q\ne 0$, $p+r \ne 0$.
\end{enumerate}
\end{itemize}
\end{Theorem}

\begin{Remark}
There are exactly six nonisomorphic three-dimensional unimodular Lie algebras and the corresponding types of three-dimensional unimodular Lie groups (see~\cite{KlepikovaSV:Milnor1976}). All of them are listed in the Table~\ref{tab_PastukhovaSV:Chibrikova2006} together with conditions on structure constants for~which the Lie algebra has this type. If there is a ``$-$'' in the Table~\ref{tab_PastukhovaSV:Chibrikova2006} at the intersection of the row, corresponding to the Lie algebra, and the column, corresponding to the type, then it means that this type of the basis  is impossible for given Lie algebra. For~the~case of Lie algebra $\mathcal{A}_1$ we give only the signs of the triple $\left(\lambda_1,\lambda_2,\lambda_3\right)$ up to reorder and sign change.
\end{Remark}

\begin{table*}[ht!]
\caption{Three-dimensional unimodular Lie algebras}
\label{tab_PastukhovaSV:Chibrikova2006}
\begin{center} \small
\begin{tabular}[c]{|c|c|c|c|c|} 
  \hline
  \multirow{2}*{Lie algebra} & \multicolumn{4}{c|}{Restrictions on the structure constants}  \\
   
  \cline{2-5}   & $\mathcal{A}_1$ & $\mathcal{A}_2$ & $\mathcal{A}_3$ & $\mathcal{A}_4$ \\
\hline
  $su(2)$            & $\left(+,+,+\right)$ & $-$ & $-$ & $-$ \\
\hline
  $sl(2,\mathbb{R})$ & $\left(+,+,-\right)$ & $\lambda_1 \ne 0$, $\lambda_2 \ne 0$ & $\lambda \ne 0$ & $\lambda_3 \ne 0$ \\
\hline
  $e(2)$             & $\left(+,+,0\right)$ & $-$ & $-$ & $-$ \\
\hline
  $e(1,1)$           & $\left(+,-,0\right)$ & \parbox[c]{80pt}{\centering $\lambda_1 = 0$, $\lambda_2 \ne 0$ \\ or \\ $\lambda_1 \ne 0$, $\lambda_2 = 0$} & $\lambda = 0$ & $\lambda_3 = 0$ \\
\hline
  $h$                & $\left(+,0,0\right)$ & $\lambda_1 = 0$, $\lambda_2 = 0$ & $-$ & $-$  \\
\hline
  $\mathbb{R}^{3}$   & $\left(0,0,0\right)$ & $-$ & $-$ & $-$  \\

  \hline
\end{tabular}
\end{center}
\end{table*}

\begin{Remark}
We note that similar bases was also constructed by G. Calvaruso, L.A.~Cordero and P.E. Parker in~\cite{KlepikovaSV:Calvaruso2007,KlepikovaSV:CP1997}.
\end{Remark}

The following classification result for the case of three-dimensional Lorentzian locally symmetric space was proved in~\cite{KlepikovaSV:Calvaruso2007}.

\begin{Theorem} \label{th:loc_sym}
A connected, simply connected three-dimensional Lorentzian locally symmetric space $(M,g)$ is locally isometric to
\begin{enumerate}
\item a Lorentzian space form $\mathbb{R}^3_1$, $\mathbb{S}^3_1$ or $\mathbb{H}^3_1$ (with zero, positive and negative sectional curvature respectively), or
\item a direct product $\mathbb{R}\times\mathbb{S}^2_1$, $\mathbb{R}\times\mathbb{H}^2_1$, ${\mathbb{S}^2\times\mathbb{R}_1}$ or $\mathbb{H}^2\times\mathbb{R}_1$, or
\item a space with a Lorentzian metric $g$, which admitted a local coordinate system $(u_1,u_2,u_3)$ such, that the metric tensor has the following form
$$ g= \begin {pmatrix} 0 & 0           & 1 \\
                       0 & \varepsilon & 0 \\
                       1 & 0           & f(u_2,u_3) \end {pmatrix},$$
where $\varepsilon=\pm 1$, $f(u_2,u_3)=u_2^2\alpha+u_2\beta(u_3)+\xi(u_3)$, $\alpha\in\mathbb{R}$, $\beta$, $\xi$ are arbitrary smooth functions.
\end{enumerate}
\end{Theorem}

\section{Three-dimensional Lorentzian Lie groups}


Further, by ``a three-dimensional Lorenzian Lie group $\left(G,\mathfrak{g}\right)$'' we shall mean a~three\nobreakdash-di\-men\-sional Lie group $G$, which equipped with a left-invariant Lorentziam metric~$g$ and having metric Lie algebra $\mathfrak{g}$. Now we can prove the following

\begin{Theorem} \label{th:A2}
A three-dimensional unimodular Lorentzian Lie group $\left(G,\mathcal{A}_2\right)$ with the~sectional curvature operator $\mathcal{K}$ exist if and only if 
\begin{enumerate}
\item $\mathcal{K}$ has the Segre type $\{111\}$ with the eigenvalues $ k_1=- k_2=- k_3 \geqslant 0$  up to renumeration, or     
\item $\mathcal{K}$ has the Segre type $\{12\}$ with the eigenvalues
\begin{enumerate} 
\item $ k_1 =  k_2 =0$, or  
\item $ k_2 <0$.  
\end{enumerate}
\end{enumerate}
\end{Theorem}

\begin{proof}
In this case the matrix of the sectional curvature operator $\mathcal{K}$ has the following form
\begin{equation*}
\mathcal{K}=\begin {pmatrix}
\frac34\lambda_1^2-\lambda_1\lambda_2 & 0 & 0 \\
0 & 2\lambda_2-\lambda_1-\frac14\lambda_1^2 & 2\lambda_2-\lambda_1 \\
0 & -2\lambda_2+\lambda_1 & -2\lambda_2+\lambda_1-\frac14\lambda_1^2
\end {pmatrix}.
\end{equation*}

If $\lambda_1=2\lambda_2$, then the matrix of the sectional curvature operator has the~diagonal form with the eigenvalues $k_1=-k_2=-k_3\geqslant0$. Else, the matrix of the sectional curvature operator $\mathcal{K}$ has the following Jordan form:
$$\mathcal{K}=\begin {pmatrix}
-\lambda_1\lambda_2+\frac34\lambda_1^2 & 0 & 0 \\
0 & -\frac14\lambda_1^2 & 1 \\
0 & 0 & -\frac14\lambda_1^2
\end {pmatrix},$$
and the eigenvalues are equal to
\begin{gather*}
 k_1= -\lambda_1\lambda_2+\frac34\lambda_1^2, \quad  k_2= -\frac14\lambda_1^2 \leqslant 0.
\end{gather*}

If $ k_2=0$, then $\lambda_1=0$ and all of the eigenvalues are equal to zero. 

Suppose that $ k_2<0$. Then, this follows that $\lambda_1 = \pm2\sqrt{- k_2}$. Expressing~$\lambda_2$, we find
$$ \lambda_2 = \mp \frac{ k_1+3 k_2}{2\sqrt{- k_2}}. $$ 
\end{proof}

The remaining cases are concerned in a similar way.

\section{Three-dimensional Lorentzian locally symmetric spaces}

The Theorem~\ref{th:loc_sym} allows us to divide the problem of studying the curvature operators on~three-dimensional locally symmetric Lorentzian manifolds by~three subtasks. At~the~same time, it is obvious that the sectional curvature operator $\mathcal{K}$ is diagonalizable for~Lorentzian manifolds of constant sectional survature $\mathbb{R}^3_1$, $\mathbb{S}^3_1$ and $\mathbb{H}^3_1$ (i.e. $\mathcal{K}$ has the Segre type~$\{111\}$) and $\mathcal{K}$ has three equal eigenvalues (zero, positive or negative respectively). 

In the case of direct products (case 2 of the Theorem~\ref{th:loc_sym}) the sectional curvature operator $\mathcal{K}$ has the Segre type~$\{111\}$ with two zero and third non-zero eigenvalues.

Therefore, only the case 3 of Theorem~\ref{th:loc_sym} is of interest, in which the~metric tensor has the following form in local coordinate system 
$$ g= \begin {pmatrix} 0 & 0           & 1 \\
                       0 & \varepsilon & 0 \\
                       1 & 0           & f(u_2,u_3) \end {pmatrix},$$
where $\varepsilon=\pm 1$, $f(u_2,u_3)=u_2^2\alpha+u_2\beta(u_3)+\xi(u_3)$, $\alpha\in\mathbb{R}$, $\beta$, $\xi$ are arbitrary smooth functions.

Calculating the matrix of the sectional curvature operator $\mathcal{K}$, we have
\begin{equation*}
\mathcal{K}= \begin {pmatrix} 0 & 0 & \frac{\alpha}{\varepsilon} \\
                              0 & 0 & 0 \\
                              0 & 0 & 0 \end {pmatrix}.
\end{equation*}
Therefore, either the sectional curvature operator $\mathcal{K}$ has the~Segre type $\{111\}$ with  the~eigenvalues $k_1=k_2=k_3=0$ for $\alpha=0$, or $\mathcal{K}$ has the Segre type~$\{12\}$ with the~eigenvalues $k_1=k_2=0$ for $\alpha\ne0$. Hence, the following theorem holds.

\begin{Theorem} \label{th:K_loc_sym}
A connected, simply connected three-dimensional Lorentzian locally symmetric space with the sectional curvature operator $\mathcal{K}$ exist if and only if
\begin{enumerate}
\item $\mathcal{K}$ has the Segre type $\{111\}$ with the equal eigenvalues, or
\item $\mathcal{K}$ has the Segre type $\{111\}$ with the two zero eigenvalues and one nonzero, or
\item $\mathcal{K}$ has the Segre type $\{12\}$ with the zero eigenvalues.
\end{enumerate} 
\end{Theorem}

\section{Sectional curvature operator of locally homogeneous Lorentzian 3-manifolds}

In this section, using the results of the previous sections, we determine under which conditions the different Segre types occur for the sectional curvature operator of a~three\nobreakdash-dimensional locally homogeneous Lorentzian manifold. 
Next theorems follows from the results on the cases of metric Lie groups and of the locally symmetric spaces.

\begin{Theorem} \label{th:loc_hom_nd}
A connected, simply connected three-dimensional Lorentzian locally homogeneous manifold $(M, g)$ with non-diagonalizable sectional curvature operator $\mathcal{K}$ exist if and only if $\mathcal{K}$ satisfies one of the following conditions:
\begin{enumerate}
\item $\mathcal{K}$ has the Segre type $\{12\}$ and
\begin{enumerate}
\item eigenvalues equal to zero, or 
\item $k_2<0$;
\end{enumerate}

\item $\mathcal{K}$ has the Segre type $\{3\}$ with negative eigenvalue;

\item $\mathcal{K}$ has the Segre type $\{1z\overline{z}\}$ and 
\begin{enumerate}
\item complex eigenvalues have negative real part, or
\item $0 \leqslant \frac{k_2+ k_3}{2} < -k_1$.
\end{enumerate}
\end{enumerate}
\end{Theorem}

%
%

\begin{Theorem}
A connected, simply connected three-dimensional Lorentzian locally homogeneous manifold $(M, g)$ with sectional curvature operator $\mathcal{K}$ with Segre type~$\{111\}$ exist if and only if eigenvalues $k_1$, $k_2$, $k_3$ satisfy one (or more than one) of the following conditions:
\begin{enumerate}
\item all eigenvalues are equal to each other;
\item two eigenvalues are equal to zero and third is nonzero;
\item exactly two of $k_1 + k_2$, $k_1 + k_3$ and $k_2 + k_3$ are equal to zero;
\item $\left(k_1 + k_2\right)\left(k_1 + k_3\right)\left(k_2 + k_3\right)<0$;
\item up to renumeration $$k_2k_3 \leqslant k_1^2<\left(\frac{k_2+k_3}{2}\right)^2\; \text{ and } \;k_1<\frac{k_2+k_3}{2};$$
\item up to renumeration $$k_2<0,\; k_3<0,\; \left|k_1\right|\leqslant\sqrt{k_2k_3};$$
\item up to renumeration $$k_1 < -\left|\frac{k_2+k_3}{2}\right|.$$
\end{enumerate}
\end{Theorem}

%

\Addresses

\end{document}